\documentclass[11pt]{amsart}
\usepackage{commath}
\usepackage{stmaryrd}
\usepackage{xcolor}
\usepackage[hidelinks]{hyperref}
\newtheorem{lemma}{Lemma}

\newtheorem{theorem}{Theorem}
\newcommand{\Teich}{\mathcal{T}}
\newcommand{\SLC}{\mathrm{SL}_2(\mathbb{C})}

\newcommand{\slc}{\mathfrak{sl}_2(\mathbb{C})}
\newcommand{\slo}{\mathfrak{sl}_2}
\newcommand{\Syst}{\mathrm{Syst}}
\newcommand{\Rep}{\mathrm{Rep}}
\newcommand{\QD}{\mathrm{QD}}
\newcommand{\RH}{\mathbf{M}}
\newcommand{\slr}{\mathfrak{sl}_2(\mathbb{R})}
\newcommand{\Belt}{\mathrm{Belt}}

\theoremstyle{definition}
\newtheorem{definition}{Definition}
\newtheorem{remark}{Remark}
\newtheorem{claim}{Claim}

\begin{document}
\title{Derivative of the Riemann-Hilbert map}
\author{Vladimir Markovi\'c and Ognjen To\v{s}i\'c}
\address{\newline Mathematical Institute  \newline University of Oxford  \newline United Kingdom}
\email{ognjen.tosic@gmail.com}
\address{\newline All Souls College  \newline University of Oxford  \newline United Kingdom}
\email{markovic@maths.ox.ac.uk}

\today


\let\johnny\thefootnote
\renewcommand{\thefootnote}{}

\footnotetext{ This work was supported by the \textsl{Simons Investigator Award}  409745 from the Simons Foundation}
\let\thefootnote\johnny

\begin{abstract}
Given a pair $(X,\nabla)$, consisting of a closed Riemann surface $X$ and a holomorphic connection $\nabla$ on the trivial principal bundle $X\times\SLC\to X$, the Riemann--Hilbert map sends $(X,\nabla)$ to its monodromy representation. We compute the derivative of this map, and provide a simple  description of  the locus where it is injective, recovering in the process several previously obtained results.
\end{abstract}
\maketitle

\section{Introduction}
By $\Sigma_g$ we denote a closed topological surface of genus $g$, and by $\Teich_g$ the Teichm\"uller space of marked closed Riemann surfaces of genus $g$. Given $X\in \Teich_g$, we let $\Omega^1(X)$ denote the space of Abelian differentials on $X$. By 
$\slo(\Omega^1(X))$ we denote the space of traceless $2\times 2$ matrices of holomorphic 1-forms on $X$.

Fix $A\in \slo(\Omega^1(X))$, and consider the trivial holomorphic principal bundle $\underline{\SLC}:=X\times\SLC\to X$ equipped with the connection $\nabla=d+A$. This is a flat holomorphic connection over a marked Riemann surface $X$  inducing the monodromy representation 
$$
\rho_A:\pi_1(\Sigma_g)\to \SLC.
$$ 
The space of $\slo$-systems over $\Sigma_g$  is defined as
$$
\Syst_g=\{(X,A): X\in \Teich_g, \,\, A\in  \slo(\Omega^1(X)), \,\, \rho_A\,\, \text{is irreducible}  \} \sslash \SLC
$$
where we identify pairs $(X,A)$ and $(X,B)$ if $A$ and $B$ are  conjugated by and element of $\SLC$.  The  Riemann--Hilbert map 
$$
\RH:\Syst_g\to\Rep_g
$$
is defined by $\RH(X,A)=\rho_A$. Here 
$$
\Rep_g=\{ \rho\in \text{Hom}\big(\pi_1(\Sigma_g) \to \SLC \big) : \rho \,\, \text{is irreducible} \}  \sslash \SLC
$$
is the character variety of irreducible representations of the surface group $\pi_1(\Sigma_g)$.

It is well known that $\RH$ is a holomorphic map between connected complex manifold of complex dimension $6g-6$ (see \cite{b-d} and \cite{c-d-h-l}).  Our main result is Theorem \ref{thm:main} below is  the explicit computation of the derivative $D\RH$.  But first, we state  its main corollary characterising the pairs $(X,A)$ where the derivative of $\RH$ is injective (and hence bijective), answering a well known question (see \cite{c-d-h-l}, and also \cite{ghys} for the broader picture).

 Denote by $\QD(X)$ the space of holomorphic quadratic differentials on $X$. Given $\alpha,\beta,\gamma\in \Omega^1(X)$, we define the subspace  $\QD(\alpha,\beta,\gamma)\leq\QD(X)$ by  
 $$
 \QD(\alpha,\beta,\gamma)=\{\alpha\varphi_1+\beta\varphi_2+\gamma\varphi_3: \varphi_1, \varphi_2,\varphi_3\in \Omega^1(X)\}.
 $$ 
\begin{theorem}\label{cor:inj}
    Let $(X, A)\in\Syst_g$, and write \[A=\begin{pmatrix}
        \alpha & \beta \\ \gamma & -\alpha
    \end{pmatrix}\]
    for $\alpha,\beta,\gamma\in \Omega^1(X)$. The derivative $D_{(X,A)}\RH$ is injective  if and only if 
    \begin{equation}\label{eq-span}
    \QD(\alpha,\beta,\gamma)=\QD(X).
    \end{equation}  
    \end{theorem}
\begin{remark}  This theorem has a number of applications which we prove in the last section. In particular,  we show how it can be used to recover theorems by  Biswas-Dumitrescu \cite{b-d},  Calsamiglia-Deroin-Heu-Loray \cite{c-d-h-l}. In Proposition 3.1 in \cite{b-d},  Biswas-Dumitrescu give a sufficient condition on the pair $(X,A)$ for the derivative  $D_{(X,A)}\RH$ to be injective. This condition  can be viewed as a cohomological version of the equality (\ref{eq-span}). 
\end{remark}

Before stating Theorem \ref{thm:main} we  describe our parameterisation of the tangent spaces to  $\Syst_g$ and $\Rep_g$.

\subsection{Beltrami differentials and tamed matrix forms}

Let  $\Belt(X)$ denote the vector space of smooth Beltrami differentials on $X$. Two differentials $\mu,\nu\in \Belt(X)$ are equivalent if 
$$
\int\limits_{X} \mu\phi=\int\limits_{X} \nu\phi
$$
for every holomorphic quadratic differential $\phi\in \QD(X)$. This is a linear equivalence relation, that we denote $\sim$, and the quotient vector space $\Belt(X)/\sim$ is naturally isomorphic to the tangent space of the Teichm\"uller space $T_X\Teich_g$.

Next, we define the notion of a tame matrix valued form. 
\begin{definition} Let $A\in \slo(\Omega^1(X))$, and $\mu\in \Belt(X)$. We say that $\dot{A}$ is $(\mu,A)$-tamed form if
$\dot{A}$  a closed $\slc$-valued 1-form whose $(0,1)$ part satisfies the equality $\dot{A}^{0,1}=\mu A$.
\end{definition}
The following lemma is elementary  and its proof is left to the reader.
\begin{lemma}\label{lemma-elem} Fix $A \in \slo(\Omega^1(X))$. Let  $\mu,\nu \in \Belt(X)$, and suppose $\dot{A}$ and $\dot{B}$ are 
$(\mu,A)$ and $(\nu,A)$ tamed respectively. If $\mu\sim \nu$ then there exists a smooth function $T:X\to \slc$ such that
$$
\dot{A}-\dot{B}-dT \in \slo(\Omega^1(X)).
$$
\end{lemma}

\subsection{Parametrisation of the tangent space to $\Syst_g$}

Fix $(X,A)\in \Syst_g$,  and define the vector space
$$
Z(X,A)=\{ (\mu,\dot{A}) : \mu \in \Belt(X), \,\, \dot{A}\,\,\text{is $(\mu,A)$-tamed} \},
$$
That is, $Z(X,A)$ is a set of pairs $(\mu, \dot{A})$ where $\mu$ is a smooth Beltrami form on $X$, and $\dot{A}$  a closed $\slc$-valued 1-form satisfying the equality $\dot{A}^{0,1}=\mu A$. We define the linear equivalence relation on $Z(X,A)$ by letting
$(\mu,\dot{A}) \sim (\nu,\dot{B})$ if the following two conditions are satisfied
\begin{enumerate} 
\item $\mu\sim \nu$,
\vskip .1cm
\item $\dot{A}^{1,0}-\dot{B}^{1,0}-\partial{T} \in [A,\slc]$, where  $T:X\to \slc$ is the function from Lemma \ref{lemma-elem}.
\end{enumerate}
\begin{remark} Whenever $\mu\sim \nu$, we have $\dot{A}^{1,0}-\dot{B}^{1,0}-\partial{T}\in \slo(\Omega^1(X))$. For the pairs $(\mu,\dot{A})$ and $(\nu,\dot{B})$ to be equivalent we require this element to live in $[A,\slc]\leq\slo(\Omega^1(X))$.
\end{remark}

It will be shown in Lemma \ref{lm:tangent-syst} that the vector space $Z(X,A)/\sim$ is naturally isomorphic to $T_{(X,A)}\Syst_g$.

\subsection{Derivative of the monodromy map $\RH$}

The tangent space $T_\rho \Rep_g$ is well-known to coincide with $H^1(\mathrm{Ad}_\rho)$, where $\mathrm{Ad}_\rho$ denotes the $\pi_1(\Sigma_g)$-module with underlying vector space $\slc$ and the action $(\gamma, T)\to\rho(\gamma)T\rho(\gamma)^{-1}$. Since $\mathrm{Ad}_\rho$ is precisely the monodromy of the flat connection 
$$
d_A:=d+\mathrm{ad}(A)
$$ 
on the bundle $\underline{\slc}:=X\times\slc\to X$, there is a de Rham isomorphism 
\begin{equation}\label{eq-rham}
\iota:H^1(X, \mathcal{E})\to H^1(\mathrm{Ad}_\rho). 
\end{equation}
Here $\mathcal{E}$ denotes the sheaf of $d_A$-flat sections of the bundle $\underline{\slc}$.
We are now ready to state our main result.
\begin{theorem}\label{thm:main} Let $(X,A)\in\Syst_g$ and $(\mu, \dot{A})\in Z(X,A)$. Then $\dot{A}$ is a $d_A$-closed ${\slc}$-valued 1-form, and hence defines a cohomology class $\chi\in H^1(X,\mathcal{E})$. Moreover, $D\RH\big([(\mu,\dot{A})]\big)=-\iota(\chi)$, where $[(\mu,\dot{A})]\in Z(X,A)/\sim $ denotes the corresponding tangent vector under the identification $T_{(X,A)}\Syst_g\approx Z(X,A)/\sim$.
\end{theorem}

\subsection{Organisation}

Theorem \ref{cor:inj} is proved in Section 2. In Section 3 we describe the tangent space to $\Syst_g$ and prove that  $T_{(X,A)}\Syst_g$ is isomorphic to $Z(X,A)/\sim$. In Section 4 we recall the description of the tangent space of $\Rep_g$ and construct the 
de Rham isomorphism (\ref{eq-rham}). In Section 5 we prove Theorem \ref{thm:main}. Applications of Theorem  \ref{cor:inj} are derived in Section 6.

\section{The kernel of  $D\RH$} 

Fix $(X,A)\in \Syst_g$, and consider a  tangent vector $(\mu,\dot{A})\in Z(X,A)$. The following three  lemmas  easily follow from Theorem \ref{thm:main}, and the isomorphism $T_{(X,A)}\Syst_g\approx Z(X,A)/\sim$. We then use these lemmas to prove Theorem \ref{cor:inj}.

\begin{lemma}\label{lemma-1} Let $(\mu,\dot{A})\in Z(X,A)$. The derivative $D\RH(\mu, \dot{A})$ vanishes if and only if there exists a smooth function
$\dot{F}:X\to\slc$  such that
\begin{equation}\label{eq-1}
\mu A=\bar{\partial}\dot{F}\,\,\,\,\,\text{ and }\,\,\,\,\,\dot{A}=\partial\dot{F}+[A,\dot{F}]+\mu A.
\end{equation}
\end{lemma}
\begin{proof}
By  Theorem \ref{thm:main} the derivative $D\RH(\mu, \dot{A})$ vanishes if and only if $\chi=[\dot{A}]\in H^1(X,\mathcal{E})$ vanishes, that is,  if $\dot{A}$ is $d_A$-exact. This means that $\dot{A}=d_A\dot{F}=d\dot{F}+[A, \dot{F}]$, for some smooth function $\dot{F}:X\to\slc$. Splitting this equality into $(0,1)$ and $(1,0)$ parts, and using that $\dot{A}^{0,1}=\mu A$, we derive (\ref{eq-1}). 

\end{proof}

 \begin{lemma}\label{lemma-2} Suppose $\mu$ is a Beltrami differential on $X$. Then there exists $\dot{A}$ such that $D\RH(\mu, \dot{A})=0$ if and only if the 1-forms $\mu\alpha,\mu\beta,\mu\gamma$ are all $\bar{\partial}$-exact, where   
 \[A=\begin{pmatrix}
        \alpha & \beta \\ \gamma & -\alpha
    \end{pmatrix}\]
\end{lemma} 
\begin{proof} If such $\dot{A}$ exists then by Lemma \ref{lemma-1} there exists  $\dot{F}:X\to\slc$ satisfying the system (\ref{eq-1}). In particular, the equality $\mu A=\bar{\partial}\dot{F}$ holds implying that the 1-forms $\mu\alpha,\mu\beta,\mu\gamma$ are all $\bar{\partial}$-exact. 

Suppose now that $\mu\alpha,\mu\beta,\mu\gamma$ are all $\bar{\partial}$-exact. Then we can construct  $\dot{F}:X\to\slc$
such that $\mu A=\bar{\partial}\dot{F}$. Set $\dot{A}=\mu A+\partial\dot{F}+[A,\dot{F}]$. One can now verify that $(\mu,\dot{A})$ solves the system (\ref{eq-1}). \end{proof}

\begin{lemma}\label{lemma-3} Suppose $\mu$ is zero in $T_X\Teich_g$. Then $D\RH(\mu, \dot{A})=0$ if and only if the pair $(\mu,\dot{A})$ represents the zero vector in $T_{(X,A)}\Syst_g$.
\end{lemma}
\begin{proof} Suppose that $D\RH(\mu, \dot{A})=0$. We need to prove that  $(\mu,\dot{A})$ represents the zero vector in $T_{(X,A)}\Syst_g$. Since $\mu$ is zero in $T_X\Teich_g$, we may assume $\mu=0$. Let $\dot{F}:X\to\slc$ be the function from Lemma \ref{lemma-1}. Then $\bar{\partial}\dot{F}=0$, and hence $\dot{F}:X\to\slc$ is holomorphic. Therefore $\dot{F}$ is constant, and by the second equality in (\ref{eq-1}) we obtain $\dot{A}=\dot{A}^{1,0}=[A, \dot{F}]$. But then 
 $\dot{A}\in [A,\slc]$,  and so $(0,\dot{A})$ vanishes in $T_{(X,A)}\Syst_g$.  \end{proof}

\subsection{Proof of Theorem \ref{cor:inj}}  
We are ready to complete the proof of Theorem \ref{cor:inj}. Let $\mu$ denote a Beltrami differential on $X$. Then by Lemma \ref{lemma-2} and Lemma \ref{lemma-3} there exists $\dot{A}$ such that $(\mu,\dot{A})$ is a non-zero vector with 
$D\RH(\mu, \dot{A})=0$ if and only if $\mu$ is not equal to zero in $T_X\Teich_g$, and  the 1-forms $\mu\alpha,\mu\beta,\mu\gamma$ are all $\bar{\partial}$-exact.  By Serre duality, the last condition is equivalent to 
$$
\int\limits_{X} \mu\phi=0
$$
for every $\phi\in \alpha \Omega^1(X)+\beta \Omega^1(X)+\gamma \Omega^1(X)\leq \QD(X)$. Thus, $D\RH$ is injective if and only if 
the equality  (\ref{eq-span}) holds.

\section{The space of $\mathfrak{sl}_2$-systems}
Recall that 
\begin{align*}
    \Syst_g=\{(X, A):X\in\Teich_g, A\in\mathfrak{sl}_2(\Omega^1(X)), \rho_A\text{ is irreducible}\}\sslash\SLC.
\end{align*}
Note that the quotient here is by the conjugation action of $\SLC$, in the sense of geometric invariant theory. Since we have restricted to the irreducible locus, this coincides with the quotient in the sense of topology. It is standard that $\Syst_g$ is a smooth complex manifold.
\par Our aim in this section is to describe the tangent space $T_{(X,A)}\Syst_g=Z(X,A)/\sim$, as described in Section 1. Given a smooth Beltrami form $\mu\in\mathrm{Belt}(X)$ with $\norm{\mu}_{L^\infty(X)}<1$, denote by $f^\mu:X\to X^\mu$ the smooth map to a Riemann surface $X^\mu$ such that 
\begin{align*}
    \bar{\partial}f^\mu=\mu\partial f^\mu.
\end{align*} 
\begin{lemma}\label{lm:tangent-syst}
    Let $(X,A)\in\Syst_g$ and $(\mu,\dot{A})\in Z(X,A)$. Then there exists a path of pairs $(X^{t\mu}, A_t)$ based at $(X, A)$, such that 
    \begin{align}\label{eq:path-exists}
        \left.\frac{d}{dt}\right|_{t=0} \left(f^{t\mu}\right)^*A_t=\dot{A}.
    \end{align}
Moreover, the map $\tau:Z(X,A)\to T_{(X, A)}\Syst_g$, defined by taking the tangent vector to this path at $t=0$, is surjective and has kernel  $B(X,A)$ consisting of pairs $(\mu,\dot{A})\in Z(X,A)$ such that 
\begin{itemize} 
\item  $\mu$ vanishes in $T_X\Teich_g$, and
\vskip .1cm
\item  $\dot{A}-dT\in[A,\slc]$ for some function $T:X\to\slc$.
\end{itemize}
\end{lemma}
We split the proof of Lemma \ref{lm:tangent-syst} into two parts. We construct the path with property (\ref{eq:path-exists}) in \S\ref{subsec:constr-path}, and then we describe the kernel of $\tau$ in \S\ref{subsec:ker-tau}.
\subsection{Constructing the path}\label{subsec:constr-path}
The path we construct will in fact be in 
\begin{align*}
    \{(X, A):X\in\Teich_g, A\in\mathfrak{sl}_2(\Omega^1(X))\}
\end{align*}
rather than $\Syst_g$. Note that this space is isomorphic to the total space of $\mathcal{H}_g\otimes\slc$, where $\mathcal{H}_g\to\Teich_g$ is the Hodge bundle of Abelian differentials over a varying Riemann surface, i.e. the fibre of $\mathcal{H}_g$ over $X\in\Teich_g$ is $\Omega^1(X)$.   
\par Let $A_t$ be a path of $\slc$ valued holomorphic 1-forms on $X^{t\mu}$, chosen so that $[\mathrm{Re}(A_t)]\in H^1(\Sigma_g,\slr)$ depends smoothly on $t$, and has $\left.\frac{d}{dt}\right|_{t=0}[\mathrm{Re}(A_t)]=[\mathrm{Re}(\dot{A})]$. By the smoothness of the Gauss--Manin connection on $\mathcal{H}_g$, it follows that $\left(f^{t\mu}\right)^*A_t$ depend smoothly on $t$, and hence we can define  
    \begin{align*}
        B= \left.\frac{d}{dt}\right|_{t=0} \left(f^{t\mu}\right)^*A_t.
    \end{align*}
    Since the 1-forms $A_t$ are holomorphic, they are also closed, and hence so is $B$. Moreover, $\mathrm{Re}(\dot{A}-B)$ is exact by construction. 
    \begin{claim}\label{claim:b}
        We have $B^{0,1}=\mu A$.
    \end{claim}
    \begin{proof}
        We work in a local coordinate $z$ on $X$, and abuse notation slightly to denote the Beltrami form by $\mu\frac{d\bar{z}}{dz}$ in this coordinate. Then $\tilde{A}_t:=\left(f^{t\mu}\right)^*A_t$ is proportional to $f^{t\mu}_z\left(dz+t\mu d\bar{z}\right)$. In particular, 
            \begin{align*}
               \tilde{A}_t^{0,1}=t\mu \tilde{A}_t^{1,0}.
            \end{align*}
        Differentiating at $t=0$, we see that $B^{0,1}=\mu A$, as desired.
    \end{proof}
    In particular, it follows that $B-\dot{A}$ is a closed form of type $(1,0)$, and is therefore holomorphic. Since $B-\dot{A}$ has exact real part, it follows by Hodge theory that $B-\dot{A}=0$. We have thus constructed a path with (\ref{eq:path-exists}).
\subsection{Surjectivity and the kernel of $\tau$}\label{subsec:ker-tau} 
We first show the surjectivity of $\tau$. Given any tangent vector $V\in T_{(X,A)}(\mathcal{H}_g\otimes\slc)$, let $(X^{t\mu},A_t)$ be a path in $\mathcal{H}_g\otimes\slc$ tangent to $V$ at $(X,A)$. Note that each $A_t$ is closed, so is $\left(f^{t\mu}\right)^*A_t$. Thus $\dot{A}:=\left.\frac{d}{dt}\right|_{t=0}\left(f^{t\mu}\right)^*A_t$ is closed, and by Claim \ref{claim:b} has $(0,1)$-part equal to $\mu A$. Thus $\dot{A}$ is $(\mu, A)$-tamed and hence $V=\tau(\mu, \dot{A})$. Thus $\tau$ is surjective.
\par We now show that $\ker\tau=B(X,A)$. Since the paths constructed in \S\ref{subsec:constr-path} lie in $\mathcal{H}_g\otimes\slc$, $\tau$ factors as 
\begin{align*}
    Z(X,A)\stackrel{\tilde{\tau}}{\longrightarrow}T_{(X,A)}(\mathcal{H}_g\otimes\slc)\to T_{(X,A)}\Syst_g.
\end{align*}
Since $\rho_A$ is irreducible, the point $(X,A)$ is stable for the $\SLC$-action on $\mathcal{H}_g\otimes\slc$. Thus $\ker\tau=\tilde{\tau}^{-1}\left((0,[A, \slc])\right)=(0, [A, \slc])+\ker\tilde{\tau}$
by standard geometric invariant theory. In the rest of this subsection, we show that 
\begin{align*}
    \ker\tilde{\tau}=\{(\mu,\dot{A})\in Z(X,A):\mu\text{ vanishes in }T_X\Teich_g\text{ and }\dot{A}\text{ is exact}\},
\end{align*} 
which implies $B(X,A)=\ker\tau$ immediately.
\par Suppose first that $(\mu,\dot{A})\in\ker\tilde{\tau}$. Then $\mu$ corresponds to the zero tangent vector in $T_X\Teich_g$. By the construction of the path in \S\ref{subsec:constr-path}, it follows that $\mathrm{Re}(\dot{A})$ is exact.  Since the period matrix of  $X^{t\mu}=X$   remains constant up to terms of order $o(t)$ (for $\mu$ is zero in the $T\Teich_g$), we see that $\mathrm{Im}(\dot{A})$ is exact as well.
   \par Conversely, if $\mu$ vanishes in $T_X\Teich_g$ and $\dot{A}=dT$, then the cohomology class $[\dot{A}]\in H^1(\Sigma_g, \slc)$ vanishes. Thus the path in \S\ref{subsec:constr-path} represents a path of holomorphic 1-forms on a fixed Riemann surface, with the same cohomology. Thus the path is constant, and $(\mu, \dot{A})\in\ker\tilde{\tau}$.

 \section{The representation variety and de Rham isomorphism}
 In this section, we explain the tangent space to the representation variety $\Rep_g$ and the de Rham isomorphism for local systems. 
\subsection{Tangent space to $\Rep_g$}\label{subsec:tangent-rep} The description of $T\Rep_g$ is much more standard than $T\Syst_g$, and can be found for instance in \cite{goldman}. We summarise the results here for the sake of completeness. 
\par Given an irreducible representation $\rho:\pi_1(\Sigma_g)\to\SLC$, the tangent space $T_\rho\Rep_g$ can be described as $H^1(\mathrm{Ad}_\rho):=Z(\mathrm{Ad}_\rho)/B(\mathrm{Ad}_\rho)$, where 
\begin{align*}
    Z(\mathrm{Ad}_\rho)=\{\eta:\pi_1(\Sigma_g)\to\slc\text{ with }\eta(\gamma_1\gamma_2)=\eta(\gamma_1)+\rho(\gamma_1)\eta(\gamma_2)\rho(\gamma_1)^{-1}\},
\end{align*} 
and 
\begin{align*}
    B(\mathrm{Ad}_\rho)=\{\rho T\rho^{-1}-T\text{ for }T\in\slc\}.
\end{align*}
Given an element $\eta\in Z(\mathrm{Ad}_\rho)$, there exists a smooth path $\rho_t:\pi_1(\Sigma_g)\to\SLC$ such that 
\begin{align*}
    \left.\frac{d}{dt}\right|_{t=0}\rho_t(\gamma)=\eta(\gamma)\rho(\gamma),
\end{align*}
and this path is tangent to $\eta+B(\mathrm{Ad}_\rho)\in H^1(\mathrm{Ad}_\rho)$.
\subsection{De Rham isomorphism for local systems}\label{subsec:derham}
There is a standard correspondence between flat vector bundles $E$ of rank $n$ over a manifold $M$, and representations $\rho:\pi_1(M)\to\mathrm{GL}_n(\mathbb{C})$ of its fundamental group. Taking the derivative of this correspondence gives an isomorphism $H^1(M, \mathcal{E})\cong H^1(\pi_1(M), \mathrm{Ad}_\rho)$, where $\mathcal{E}$ is the sheaf of parallel sections of $E$. In this section, we describe this isomorphism explicitly. We omit the proofs as they are standard application of homological algebra. A much more general result can be found in Proposition 6.3 in \cite{g-m}.
\par We first describe how $H^1(M,\mathcal{E})$ and $H^1(\mathrm{Ad}_\rho)$ can be understood explicitly. Consider the de Rham resolution of $\mathcal{E}$
\begin{align*}
    \mathcal{E}\to E\stackrel{d^\nabla}{\longrightarrow} E\otimes T^*M\stackrel{d^\nabla}{\longrightarrow} E\otimes\bigwedge\nolimits^2 T^*M.
\end{align*}
This is a soft resolution, and its sections can hence be used to compute the cohomology 
\begin{align*}
    H^1(\Sigma_g,\mathcal{E})=\frac{\{A\in C^\infty(E\otimes T^*\Sigma_g):d^\nabla A=0\}}{\{d^\nabla B: B\in C^\infty(E)\}}.
\end{align*} Also recall that $H^1(\mathrm{Ad}_\rho)=Z(\mathrm{Ad}_\rho)/B(\mathrm{Ad}_\rho)$, where 
$$
Z(\mathrm{Ad}_\rho)=\{\eta:\pi_1(\Sigma_g)\to\mathbb{C}^n\text{ such that }\eta(\gamma_1\gamma_2)=\eta(\gamma_1)+\rho(\gamma_1)\eta(\gamma_2)\rho(\gamma_1)^{-1}\},
$$
and $$
B(\mathrm{Ad}_\rho)=\{\rho(\gamma)\eta\rho(\gamma)^{-1}-\eta:\eta\in\mathbb{C}^n\}.
$$
\par We now describe the de Rham isomorphism $\iota:H^1(M,\mathcal{E})\to H^1(\mathrm{Ad}_\rho)$. Let $p:\tilde{M}\to M$ be the universal cover, and $\tilde{E}=p^*E$ be the pullback of $E$. Note that the deck group action of $\pi_1(M)$ on $\tilde{M}$ naturally extends to an action on $\tilde{E}$. Fix an arbitrary basepoint $x\in \tilde{M}$.
\par Given a $d^\nabla$-closed $E$-valued 1-form $A$, the class $\iota([A])$ can be computed as follows. Since $\tilde{E}$ is the trivial flat bundle, it admits a section $T$ such that $dT=p^*A$. It is easily seen that $\gamma T-T$ is parallel for any $\gamma\in\pi_1(M)$, and hence 
\begin{align*}
    \gamma\longrightarrow (\gamma T-T)(x)
\end{align*}
defines a cycle in $Z(\mathrm{Ad}_\rho)$. The image of this cycle in $H^1(\mathrm{Ad}_\rho)$ is $\iota([\chi])$.
\section{Proof of Theorem \ref{thm:main}}
\subsection{Trivializing the flat $\SLC$-bundle}
We now describe the setup in which we will prove Theorem \ref{thm:main}.
\par Let $X$ be a Riemann surface and $A\in \mathfrak{sl}_2(\Omega^1(X))$. Let $\rho:\pi_1(X)\to\SLC$ be the monodromy of $d+A$. The following lemma shows how to trivialize $(\underline{\SLC}, d+A)$. Denote by $\tilde{X}$ the universal cover of $X$. This lemma is entirely contained in the literature (see e.g. the introduction to \cite{c-d-h-l}), but we include a proof for completeness.
\begin{lemma}\label{lm:trivialize}
    There exists a map $F:\tilde{X}\to\SLC$ such that 
    \begin{align*}
        dF+AF=0\text{ and }F(\gamma x)=F(x)\rho(\gamma)^{-1}.
    \end{align*}
    Then the map $\Phi:\tilde{X}\times\SLC \to \tilde{X}\times\SLC $ defined by $\Phi(x, T)=(x, F(x)T)$ is an isomorphism of principal bundles with the following two properties
    \begin{enumerate}
        \item $\Phi$ conjugates the $\pi_1(X)$-action $\gamma\cdot(x,T)=(\gamma x,\rho(\gamma)T)$ to the action $\gamma\cdot(x,T)=(\gamma x, T)$, and 
        \item $\Phi^*(d+A)=d$. 
    \end{enumerate} 
\end{lemma}
\begin{proof}
    Fix an arbitrary basepoint $\tilde{x}_0\in\tilde{X}$, and let $F_0\in\SLC$ be such that $\rho$ is the monodromy of $d+A$ with basepoint $\tilde{x}_0$ relative to the basis that consists of the columns of $F_0$. Define $F(x)$ as follows: for $i=1,2$, its $i$-th column is the $d_A$-parallel transport of the $i$-th column of $F_0$ along an arbitrary path from $\tilde{x}_0$ to $x$. It is immediate that $dF+AF=0$, since $F$ is obtained through parallel transport. Since the monodromy of $d+A$ is precisely $\rho$, it follows that $F(x)=F(\gamma x)\rho(\gamma)$. 
    \par Moreover, if we let $\sigma:\tilde{X}\to\SLC$ be a section of the bundle $(\underline{\SLC}, d_A)$, then
    \begin{align*}
        \nabla(F^{-1}\sigma)&=\sigma^{-1}F d(F^{-1}\sigma)\\
        &=\sigma^{-1}F\left(-F^{-1}(dF) F^{-1}\sigma+F^{-1}d\sigma\right)\\
        &=\sigma^{-1}d\sigma+\sigma^{-1}A\sigma=d_A\sigma,
    \end{align*}
    and the result is shown. 
\end{proof}
\subsection{Computing the derivative}
We now let $(X^{t\mu}, A_t)$ be the path at $(X,A)\in\mathrm{Syst}_g$ tangent to the vector $(\mu,\dot{A})$ with $\dot{A}^{0,1}=\mu A$, as in Lemma \ref{lm:tangent-syst}. Let $F_t:\tilde{X}^{t\mu}\to \mathrm{SL}_2(\mathbb{C})$ be as in Lemma \ref{lm:trivialize}, and suppose that the monodromy of $d+A_t$ is $\rho_t:\pi_1(X^{t\mu})\to\mathrm{SL}_2(\mathbb{C})$. 
\par We use $f^{t\mu}$ to transport $A_t, F_t$ and $\rho_t$ to $X$. Recall that $\left.\frac{d}{dt}\right|_{t=0}A_t=\dot{A}$. We have 
\begin{align*}
    d{F}_t+{A}_t{F}_t=0\text{ and }{F}_t(\gamma x)={F}_t(x)\cdot{\rho}_t(\gamma)^{-1}.
\end{align*}
Let ${F}_t= (\mathrm{id}+t\dot{F}+o(t))\cdot F$ for some map $\dot{F}:\tilde{X}\to\mathfrak{sl}_2(\mathbb{C})$. Then 
\begin{align*}
    d_A\dot{F}+\dot{A}=d\dot{F}+\dot{A}+[A,\dot{F}]=0.
\end{align*}
We thus have $d_A\dot{F}=-\dot{A}$. Similarly, if we set ${\rho}_t(\gamma)=(\mathrm{id}+t\dot{\rho}(\gamma)+o(t))\cdot\rho(\gamma)$, we have
\begin{align*}
    \dot{\rho}(\gamma)&=F(\gamma x)^{-1}\left(\dot{F}(x)-\dot{F}(\gamma x)\right) F(\gamma x).
\end{align*} 
Set $B=F^{-1}\dot{F}F$. Then $\dot{\rho}(\gamma)=\rho(\gamma)B(x)\rho(\gamma)^{-1}-B(\gamma x)$. It is easy to see that $\rho(\gamma)B(x)\rho(\gamma)^{-1}-B(\gamma x)$ does not depend on $x\in\tilde{X}$, so we have 
\begin{align*}
    \dot{\rho}(\gamma)=\rho(\gamma)B(\gamma^{-1}x)\rho(\gamma)^{-1}-B(x).
\end{align*}
From the description of the de Rham isomorphism $H^1(X,\mathcal{E})\to H^1(\mathrm{Ad}_\rho)$ in \S\ref{subsec:derham}, we see that $[\dot{\rho}]=-\iota[\dot{A}]$.

\section{Applications} Recall that Theorem \ref{cor:inj} states that the derivative $D\RH$ is injective on the tangent space 
$T_{(X,A)}\Syst_g$ if and only if $\QD(\alpha,\beta,\gamma)=\QD(X)$,  where 
 \[A=\begin{pmatrix}
        \alpha & \beta \\ \gamma & -\alpha.
    \end{pmatrix}\]
    In this section we derive several applications of this theorem.

    \subsection{Preliminaries}  
  
 Let $X$ be a  Riemann surface of genus at least two, and let $\xi_i\in \Omega^1(X)$, $i=1,...,k$. Let
 $\QD(\xi_1,...,\xi_k)\leq\QD(X)$ be the vector subspace generated by the products $\xi_i\varphi$, $i=1,...,k$, and $\varphi\in \Omega^1(X)$.  The following was proved by Rauch (see  statement C in \cite{rauch}).

\begin{theorem}\label{thm-rauch} Let $X$ be a hyperelliptic Riemann surface of genus at least two, and suppose 
$\xi_1,...,\xi_g$ is a basis of $\Omega^1(X)$. Then $\text{dim}\big(\QD(\xi_1,...,\xi_g)\big)=2g-1$. 
\end{theorem}

The next theorem is a classical result.

\begin{theorem}\label{thm-noether} Suppose $g\ge 3$. There exists a dense open subset $U\subset \mathcal{H}_g^{\oplus 3}$
such that the equality $\QD(\alpha,\beta,\gamma)=\QD(X)$ holds for each $(\alpha,\beta,\gamma)\in U$.
\end{theorem}
\begin{remark}
Let $p:\mathcal{H}_g^{\oplus 3}\to\Teich_g$ be the natural projection. According to the result of \cite{gieseker}, the set $U$ from Theorem \ref{thm-noether} can in fact be taken such that $U\cap p^{-1}(X)$ is
dense in $p^{-1}(X)$ for any non-hyperelliptic $X\in\Teich_g$. The same result can be recovered from the well known Max Noether theorem (see III.11.20 in \cite{f-k}).
\end{remark}
\begin{proof} Let $X$ be a non-hyperelliptic Riemann surface. By Theorem 1.1 in \cite{gieseker},
the set of triples $(\alpha,\beta,\gamma)\in \big(\Omega^1(X)\big)^{3}$ such that $\QD(\alpha,\beta,\gamma)=\QD(X)$ is an open and dense subset of $\big(\Omega^1(X)\big)^{3}$. Taking the union of these dense open subsets over all non-hyperelliptic Riemann surface in $\Teich_g$ yields the required open and dense subset of $\mathcal{H}_g^{\oplus 3}$.
\end{proof}

\subsection{Two corollaries}

The following is a corollary of Theorem \ref{cor:inj} and Theorem \ref{thm-rauch}. 

\begin{theorem}\label{cor-1} Suppose $X\in \Teich_g$ is hyperelliptic and let $(X,A)\in \Syst_g$.  Then $D\RH$ is injective on the tangent space $T_{(X,A)}\Syst_g$ if and only if $g=2$.
\end{theorem}
\begin{remark} That  $D\RH$ is injective on the tangent space $T_{(X,A)}\Syst_g$ if  $g=2$ is a well known theorem by  Calsamiglia-Deroin-Heu-Loray \cite{c-d-h-l}.
\end{remark}
\begin{proof} Assume at first that $g>2$. Then by Theorem \ref{thm-rauch} the following inequality holds
$$ 
\text{dim} \big(\QD(\alpha,\beta,\gamma)\big)\le 2g-1<3g-3=\text{dim}\big(\QD(X)\big). 
$$
Thus, the equality (\ref{eq-span}) can never hold on such $X$ regardless of the choice of Abelian differentials $\alpha,\beta,\gamma$.

We now consider the case $g=2$. Since $(X,A)\in \Syst_g$, we know that the monodromy $\rho_A:\pi_1(\Sigma_g)\to \SLC$ is irreducible. Then the three Abelian differentials $\alpha,\beta$, and $\gamma$, span $\Omega^1(X)$. This was observed by Biswas-Dumitrescu (see the proof of Proposition 4.1 in \cite{b-d}). Since $\text{dim}(\QD(X))=2$, without  loss of generality we may assume that $\alpha$ and $\beta$ span $\Omega^1(X)$. But then by Theorem \ref{thm-rauch} we know that the three quadratic differentials $\alpha^2, \alpha\beta, \beta^2$ represent a basis for $\QD(X)$. By Theorem \ref{cor:inj} the derivative $D\RH$ is injective on $T_{(X,A)}\Syst_g$. \end{proof}

In the remainder of this section we let 
$$
I: \mathcal{H}_g^{\oplus 3} \to \mathcal{H}_g\otimes \slc
$$
denote the identification 
 $$
 (\alpha,\beta,\gamma) \,\,\, \mapsto \,\,\, 
 A=\begin{pmatrix}
        \alpha & \beta \\ \gamma & -\alpha
    \end{pmatrix}.     
  $$  

\begin{theorem}\label{cor-2} The derivative $D\RH$ is locally injective on an open and dense subset of $\Syst_g$.
\end{theorem}
\begin{remark} This is a recent theorem by Biswas-Dumitrescu \cite{b-d}.
\end{remark}
\begin{proof} Let $U\subset \mathcal{H}_g^{\oplus 3}$ be the open dense set from Theorem \ref{thm-noether}. Moreover, let $V$ be the open dense subset of $\mathcal{H}_g\otimes\slc$ that corresponds to irreducible flat connections. Then $I(U)\cap V$ is a dense open subset of $\mathcal{H}_g\otimes\slc$. Since the quotient map $q:V\to V\sslash \SLC$ is open, it follows that $q(I(U)\cap V)$ is the desired dense open set by Theorem  \ref{cor:inj}.
\end{proof}

\end{document}